\documentclass{amsart}
\usepackage{amssymb}
\usepackage{hyperref}
\newtheorem{theorem}{Theorem}[section]
\newtheorem{lemma}[theorem]{Lemma}
\numberwithin{equation}{section}
\DeclareMathOperator{\supp}{supp}

\title[Low dimensional pinned distance sets]{Low dimensional pinned distance sets via spherical averages}
\author[T.~L.~J.~Harris]{Terence L.~J.~Harris}
\address{Department of Mathematics, Cornell University, Ithaca, NY 14853-4201, USA}
\email{tlh236@cornell.edu}

\begin{document}
\begin{abstract} \begin{sloppypar} An inequality is derived for the average energies of pinned distance measures. This refines Mattila's theorem on distance sets \cite{mattila} to pinned distance sets, and gives an analogue of Liu's theorem~\cite{liu1} for pinned distance sets of dimension smaller than 1. \end{sloppypar} \end{abstract}
\maketitle 
\section{Introduction}
For $n \geq 2$ let $d: \mathbb{R}^n \times \mathbb{R}^n \to [0, \infty)$ be the Euclidean distance function $(x,y) \mapsto \lvert x-y \rvert$, and for fixed $x \in \mathbb{R}^n$ let $d_x : \mathbb{R}^n \to [0, \infty)$ be the pinned distance function $d_x(y) = d(x,y)$. For finite compactly supported Borel measures $\mu$ and $\nu$ on $\mathbb{R}^n$, define $\Delta(\mu,\nu)$ by
\begin{align*} \notag \int \phi(s) \, d\Delta(\mu,\nu)(s) &= \int s^{-(n-1)/2} \phi(s) \, d_{\#}(\mu,\nu)(s) \\
 &= \int \int |x-y|^{-(n-1)/2} \phi\left( |x-y| \right) \, d\mu(x) \, d\nu(y), \end{align*} 
for all non-negative Borel measurable $\phi: \mathbb{R} \to [0, \infty]$. Let
\[  \sigma(\mu, \nu)(r) = \int_{S^{n-1}} \widehat{\mu(r\xi) } \overline{ \widehat{ \nu(r \xi ) } } \, d\sigma(\xi), \quad r \geq 0,  \]
where $\sigma$ is the surface measure on the sphere $S^{n-1}$, and 
\[ \widehat{\mu}(\xi) = \int e^{-2\pi i \langle x, \xi \rangle } \, d\mu(x). \]
 The aim of this work is to give a sufficient condition for the inequality
\begin{equation} \label{energyinequality} \int I_t( \Delta(\nu, \delta_x )) \, d\mu(x) \lesssim_{t, \alpha,\gamma} c_{\alpha}(\mu) I_{\gamma}(\nu), \end{equation}
in terms of $t$, $\alpha$ and $\gamma$, via the $L^2$ spherical averages of $\widehat{\mu}$. Here 
\[ c_{\alpha}(\mu) = \sup_{\substack{x \in \mathbb{R}^n \\ r >0 } } \frac{ \mu(B(x,r))}{r^{\alpha}}, \]
the mutual $t$-energy of $\mu$ and $\nu$ is defined by 
\[ I_t(\mu, \nu) = \int \int |x-y|^{-t} \, d\mu(x) \, d\nu(y), \]
 and $I_t(\mu) := I_t(\mu,\mu)$. The main result, Theorem~\ref{pinned}, refines Mattila's theorem (see \cite[Theorem~4.16 and Theorem~4.17]{mattila} or \cite[Proposition~15.2]{mattila2}) to pinned distance measures, and is the Hausdorff dimension analogue to Liu's theorem~\cite{liu1}, which gives a sufficient condition for the support of some $d_{x \#}(\mu)$ to have positive Lebesgue measure. The proof consists of augmenting Mattila's proof with a ``Cauchy-Schwarz reversal'' technique (also used in \cite{liu1}). I do not know if the proof from \cite{liu1} can be directly extended to lower dimensional pinned distance sets. The two main tools used in \cite{liu1} are Lemma~\ref{csreversal} (stated below) and Liu's identity \cite[Theorem~1.9]{liu1}:
\begin{equation} \label{liuidentity} \int_0^{\infty} \left\lvert \left(\sigma_t \ast f\right)(x) \right\rvert^2 t^{n-1} \, dt  = \int_0^{\infty} \left\lvert \left( \widehat{ \sigma_r } \ast f \right)(x) \right\rvert^2 r^{n-1} \, dr, \quad x \in \mathbb{R}^n,  \end{equation}
for Schwartz $f$ on $\mathbb{R}^n$, where $\sigma_r$ is the pushforward of $\sigma$ under the map $\xi \mapsto r \xi$. One key part of the proof of Mattila's theorem is the identity \cite[Theorem 4.6]{mattila}
\begin{equation} \label{mattilaidentity} \int_0^{\infty} \Delta(f,g)(t) \Delta(h,k)(t) \, dt  =  c_n\int_0^{\infty} \Sigma(f,g)(r) \Sigma(h,k)(r) \, dr,  \end{equation}
for Schwartz $f,g,h,k$ on $\mathbb{R}^n$, where 
\[ \Sigma(f, g)(r) = r^{(n-1)/2}\sigma(f, g)(r), \quad r \geq 0. \]
The identity \eqref{liuidentity} can be formally obtained from \eqref{mattilaidentity} by setting $h=\overline{f}$ and $g=k = \delta_x$, and then setting $f$ equal to a Gaussian to obtain that $c_n =1$. In the original application Mattila used $g= f$ and $h=k = \overline{f}$. This suggests that using $g = k = \delta_x$ in Mattila's original proof should result in a pinned distance version, and this is the main idea guiding the proof of Theorem~\ref{pinned} below.

The case $\mu = \nu$ in \eqref{energyinequality} is related to the distance set problem, which asks whether the condition $\dim A > n/2$ for $A \subseteq \mathbb{R}^n$ implies that $d(A \times A)$ has positive Lebesgue measure. This is still open, as is the stronger pinned version, which asks whether $\dim A > n/2$ implies the existence of some $x \in A$ such that $d_x(A)$ has positive measure. 

Throughout, given $\alpha \in [0,n]$ let $\beta\left(\alpha, S^{n-1} \right)$ be the supremum over all $\beta \geq 0$ such that for all Borel measures $\mu$ with $\supp \mu \subseteq B(0,1)$,
\[ \int \left\lvert \widehat{\mu}(r \xi ) \right\rvert^2 \, d\sigma(\xi) \lesssim_{\beta} r^{-\beta} I_{\alpha}(\mu), \quad \forall r > 0. \]
The following ``Cauchy-Schwarz reversal'' lemma will be needed. In \cite[p.~197]{mattila2} it is attributed to \cite[Lemma~C.1]{bbcr}; the version below is identical to Lemma~3.2 from \cite{liu1}.
\begin{lemma} \label{csreversal} Let $\mu$ be a Borel measure supported in the unit ball of $\mathbb{R}^n$. Then for any $\alpha \in [0,n]$ and any $\epsilon >0$,
 \[ \int \left|\left( \widehat{ \sigma_r} \ast f \right)(x)  \right|^2 \, d\mu(x)  \lesssim_{\alpha, \epsilon} c_{\alpha}(\mu) r^{\epsilon - \beta\left(\alpha, S^{n-1}\right)} \int \left| \widehat{f}(r\xi) \right|^2 \, d\sigma(\xi), \quad \forall r \geq 1, \]
for all Schwartz $f$. 
\end{lemma}

\section{Average energies of pinned distance measures}
The following theorem is the main result. 
\begin{theorem} \label{pinned} Let $n \geq 2$. If $\alpha, \gamma \in [0,n]$ and
\begin{equation} \label{tassumption}  0 < t < \gamma+\beta\left(\alpha, S^{n-1} \right) - n+1 \leq 1, \end{equation}
then for any Borel measures $\mu$ and $\nu$ supported in the unit ball of $\mathbb{R}^n$,
\begin{equation} \label{energybound}
\begin{aligned} 
\int I_t( d_{\#}(\nu, \delta_x)) \, d\mu(x) &\lesssim_{t, \alpha,\gamma} c_{\alpha}(\mu) I_{\gamma}(\nu), \quad &\gamma \leq (n-1)/2, \\
\int I_t( \Delta(\nu, \delta_x )) \, d\mu(x) &\lesssim_{t, \alpha,\gamma} c_{\alpha}(\mu) I_{\gamma}(\nu), \quad &\gamma > (n-1)/2. \end{aligned} \end{equation} \end{theorem}

\begin{sloppypar} One corollary of Theorem~\ref{pinned} is that for any Borel set $A$ with $(n-1)/2 < \dim A \leq n/2$, 
\begin{equation} \label{oberlinoberlin} \forall \epsilon >0 \, \exists  x \in A: \,  \dim d_x(A) \geq \dim A - \frac{(n-1)}{2} - \epsilon. \end{equation}
This is originally due to Oberlin-Oberlin \cite{oberlinK}, and follows from \eqref{energybound}, Frostman's lemma and the lower bound $\beta(\alpha, S^{n-1}) \geq (n-1)/2$ for $\alpha \geq (n-1)/2$ \cite{mattila}. The lower bound in \eqref{oberlinoberlin} for the full distance set was proved earlier by Falconer \cite{falconer}. For $n \in \{2,3\}$, Shmerkin's bound from \cite{shmerkin} gives an improvement over \eqref{oberlinoberlin}, and therefore \eqref{oberlinoberlin} is not sharp for $n \in \{2,3\}$ and $(n-1)/2 < \dim A \leq n/2$. \end{sloppypar}

For $\dim A > n/2$, another corollary of Theorem~\ref{pinned} is that 
\begin{equation} \label{etcetc} \forall \epsilon >0 \, \exists  x \in A: \, \dim d_x(A) \geq \min\left\{ 1, 2 \dim A-\frac{\dim A}{n} - (n-1) \right\}-\epsilon, \end{equation}
which follows from \eqref{energybound} and the following inequality from \cite{duzhang}:
\[ \beta\left(\alpha, S^{n-1} \right) \geq \frac{\alpha(n-1)}{n}, \quad n/2 \leq \alpha \leq n. \]
When $n =2$, \eqref{etcetc} is weaker than the combined results of \cite{liu2, shmerkin2, shmerkin} for all values of $\dim A$. For even $n \geq 4$, \eqref{etcetc} is weaker than what would likely follow from the methods in \cite{DIOWZ,liu2}. For odd $n \geq 3$, and 
\[ \frac{n}{2} < \dim A \leq \frac{n}{2} + \frac{1}{4} + \frac{1}{8n-4}, \]
 \eqref{etcetc} is new (as far as I am aware), with the exception that if $n=3$ and $\dim A$ is sufficiently close to $3/2$, the bound from \cite{shmerkin} is better than \eqref{etcetc} by a small absolute constant.    
\begin{proof}[Proof of Theorem~\ref{pinned}] Assume that $\gamma > (n-1)/2$; the proof for the case $\gamma \leq (n-1)/2$ is virtually identical with one simplification. By scaling it may be assumed that $\mu$ and $\nu$ are probability measures. Let $f$ and $g$ be non-negative smooth compactly supported functions on $\mathbb{R}^n$, and abbreviate $\Delta(f,g) = \Delta$, which is a finite measure by \cite[Lemma~4.3]{mattila}. Then 
\begin{equation} \label{bb1} I_t( \Delta ) \lesssim |I_t(F) | + |I_t(\Delta, K)| \quad \text{(\cite[p.~221]{mattila})}. \end{equation}
The functions $F$, $K$ and corresponding quantities $I_t(F)$, $I_t(\Delta,K)$ from \cite{mattila} will not be redefined here; all that will be needed is that $F$ satisfies
 \begin{align} \label{bb2} |I_t(F)| &\lesssim \int_0^{\infty} r^{t+n-2} \left\lvert \sigma(f,g)(r) \right\rvert^2 \, dr \quad \text{(\cite[p.~221]{mattila})} \\
\notag &\lesssim \left\lVert f \right\rVert_1^2\left\lVert g \right\rVert_1^2 + \int_1^{\infty} r^{t+n-2} \left\lvert \sigma(f,g)(r) \right\rvert^2 \, dr, \end{align}
and that $K$ satisfies both
\begin{equation} \label{bb3}  \left\lvert I_t(\Delta, K) \right\rvert = \left\lvert \int_0^{\infty} \int_0^{\infty} \Delta(s) K(x) |s-x|^{-t} \, ds \, dx\right\rvert \quad \text{(\cite[p.~221]{mattila})}, \end{equation}
and
\begin{equation} \label{bb4} K(x)  = x^{1/2} \int_0^{\infty} r^{n/2} R(rx) \sigma(f,g)(r) \, dr , \quad x>0,  \quad \text{(\cite[Lemma~4.14]{mattila})},\end{equation}
where $R: (0, \infty) \to \mathbb{C}$ is some Borel function with 
\begin{equation} \label{bb5} \left\lvert R(x) \right\rvert \lesssim \min\left\{ x^{-1/2}, x^{-3/2} \right\}, \quad \forall x >0 \quad \text{(\cite[Lemma~4.11]{mattila}).} \end{equation}
Interpolating the two bounds in \eqref{bb5} gives 
\begin{equation} \label{nbb6} \left\lvert R(x) \right\rvert \lesssim x^{(t-3)/2},\quad \forall x >0. \end{equation}
Let $\epsilon >0$ be small, to be chosen later. Using \eqref{bb3}, \eqref{bb4} and then \eqref{nbb6} gives 
\begin{align} \notag \left\lvert I_t(\Delta, K) \right\rvert &= \left\lvert \int_0^{\infty} \Delta(s) \int_0^{\infty} r^{n/2} \sigma(f,g)(r) \int_0^{\infty}  |s-x|^{-t} x^{1/2}  R(rx)  \, dx \, dr \, ds  \right\rvert \\
\notag &\lesssim  \int_0^{\infty} \Delta(s) \int_0^{\infty} r^{(n+t-3)/2} \left\lvert\sigma(f,g)(r)\right\rvert \int_0^{\infty}  |s-x|^{-t} x^{(t-2)/2} \, dx \, dr \, ds  \\
\label{3rdlast} &\sim \left(\int_0^{\infty} s^{-t/2}\Delta(s) \, ds\right)\cdot \left( \int_0^{\infty} r^{(n+t-3)/2} \left\lvert\sigma(f,g)(r) \right\rvert  \, dr  \right)\\
\label{2ndlast} &\lesssim \left( \int_0^{\infty} r^{(n+t-3)/2} \left\lvert \sigma(f,g)(r)\right\vert  \, dr  \right)^2 \\
\label{Kbound} &\lesssim_{\epsilon} \left\lVert f \right\rVert_1^2\left\lVert g \right\rVert_1^2 + \int_1^{\infty} r^{n+t-2+\epsilon} \left\lvert \sigma(f,g)(r) \right\rvert^2 \, dr. \end{align} 
 To get from \eqref{3rdlast} to \eqref{2ndlast}, the term $\int_0^{\infty} s^{-t/2}\Delta(s) \, ds$ is equal to $I_{(n+t-1)/2}(f,g)$ by definition, which equals a constant multiple of $\int_0^{\infty} r^{(n+t-3)/2} \sigma(f,g)(r)  \, dr$ by polar coordinates and the Fourier formula for the mutual energy \cite[Eq.~3.5]{mattila}. 

\begin{sloppypar} Substituting \eqref{bb2} and  \eqref{Kbound} into  \eqref{bb1} gives
\begin{equation} \label{tool} I_t( \Delta(f,g) ) \lesssim_{\epsilon} \left\lVert f \right\rVert_1^2\left\lVert g \right\rVert_1^2+ \int_1^{\infty} r^{n+t-2+\epsilon} \left\lvert \sigma(f,g)(r) \right\vert^2 \, dr, \end{equation}
for any smooth, non-negative compactly supported functions $f$ and $g$. For each integer $j \geq 1$ let $\phi_j(z) = j^n\phi(jz)$, where $\phi$ is a non-negative radial bump function on $\mathbb{R}^n$ which is compactly supported in the unit ball and satisfies $\int \phi = 1$. For fixed $x \in \supp \mu$, taking $f =f_j =  \nu \ast \phi_j$ and $g =g_{j,x} =  \delta_x \ast \phi_j$ in \eqref{tool} gives
\begin{multline} \label{firstline} I_t(\Delta(\nu, \delta_x ) ) \leq \liminf_{j \to \infty} I_t\left( \Delta\left(f_j,g_{j,x}\right) \right) \\
\lesssim_{\epsilon}  1+\liminf_{j \to \infty } \int_1^{\infty} r^{n+t-2+\epsilon} \left\lvert \sigma\left(f_j,g_{j,x}\right)(r) \right\vert^2 \, dr. \end{multline}
The first inequality in \eqref{firstline} is justified by the following argument. If $\nu_m := \nu\restriction{A_m}$ where $A_m = \left\{ y \in \supp \nu: \int \left\lvert y-z \right\rvert^{-\gamma} \, d\nu(z) \leq m \right\}$,
then $c_{\gamma}\left( \nu_m \right) < \infty$ (see \cite[p.~20]{mattila2}), so by using 
\[ I_t(\Delta(\nu, \delta_x ) ) = \lim_{m \to \infty} I_t(\Delta(\nu_m, \delta_x ) ), \qquad I_t(\Delta(\nu_m \ast \phi_j, g_{j,x} ) ) \leq I_t(\Delta(f_j, g_{j,x} ) ), \]
it may be assumed in proving the first part of \eqref{firstline} that $c_{\gamma}(\nu) < \infty$. By \cite[Theorem~2.8]{billingsley} $f_j \times g_{j,x} \stackrel{*}{\rightharpoonup} \nu \times \delta_x$, and hence $d_{\#}(f_j \times g_{j,x}) \stackrel{*}{\rightharpoonup} d_{\#}(\nu \times \delta_x)$. Using the definition of $\Delta$, Fubini's theorem, and the assumption that $c_{\gamma}(\nu) < \infty$, gives 
\begin{equation} \label{excludeorigin} \Delta\left(f_j \times g_{j,x}\right)\left[ 0, \varepsilon \right] \lesssim c_{\gamma}(\nu)\varepsilon^{\gamma - (n-1)/2}, \quad \forall \varepsilon >0, \end{equation}
and similarly
\begin{equation}  \label{excludeorigin2} \Delta\left(\nu \times \delta_x\right)\left[ 0, \varepsilon \right] \lesssim c_{\gamma}(\nu)\varepsilon^{\gamma - (n-1)/2}, \quad \forall \varepsilon >0. \end{equation}
The condition $d_{\#}(f_j \times g_{j,x}) \stackrel{*}{\rightharpoonup} d_{\#}(\nu \times \delta_x)$ combined with \eqref{excludeorigin} and \eqref{excludeorigin2} yields $\Delta(f_j \times g_{j,x}) \stackrel{*}{\rightharpoonup} \Delta(\nu \times \delta_x)$. Therefore the Fourier transform of $\Delta(f_j, g_{j,x})$ converges pointwise to the Fourier transform of $\Delta(\nu, \delta_x )$. Fatou's lemma and the Fourier formula for energy (\cite[Eq.~3.5]{mattila}) then give the first part of \eqref{firstline}. \end{sloppypar}

 Integrating the outer parts of \eqref{firstline} with respect to $\mu$ gives 
\begin{multline} \label{dct} \int I_t(\Delta(\nu, \delta_x ) ) \, d\mu(x)   \\
\lesssim_{\epsilon}  1+  \liminf_{j \to \infty } \int_1^{\infty}  \int r^{n+t-2+\epsilon} \left\lvert \sigma\left(f_j,g_{j,x}\right)(r) \right\vert^2 \, d\mu(x) \, dr. \end{multline}
The inequality $\left\lvert \sigma\left(f_j,g_{j,x}\right)(r) \right\vert \leq \left\lvert \sigma\left(\nu,\delta_x\right)(r) \right\vert$ holds for every $r>0$ and $x \in \supp \mu$, since $\phi$ is a radial bump function which integrates to 1. Moreover, $\sigma\left(f_j,g_{j,x}\right)(r) \to \sigma(\nu, \delta_x)(r)$ as $j \to \infty$, pointwise for every $r>0$ and $x \in \supp \mu$. Therefore applying dominated convergence and then Lemma~\ref{csreversal} to \eqref{dct} gives 
\begin{align*} \int I_t(\Delta(\nu, \delta_x ) ) \, d\mu(x)   &\lesssim_{\epsilon} 1+   \int_1^{\infty}  \int r^{n+t-2+\epsilon} \left\lvert \sigma\left(\nu,\delta_x \right)(r) \right\vert^2 \, d\mu(x) \, dr \\
&= 1 + \int_1^{\infty} r^{n+t-2+\epsilon} \int \left\lvert \left(\widehat{\sigma_r} \ast \nu \right)(x)  \right\vert^2 \, d\mu(x) \, dr \\
 & \lesssim_{\epsilon} 1 + c_{\alpha}(\mu)\int_1^{\infty} r^{n+t-2+2\epsilon- \beta\left(\alpha, S^{n-1}\right)} \int \left\lvert \widehat{\nu}(r\xi) \right\rvert^2 \, d\sigma(\xi) \, dr \\
& \lesssim c_{\alpha}(\mu)I_{\gamma}(\nu),  \end{align*}
by \eqref{tassumption}, provided $\epsilon$ is small enough, again using polar coordinates and the Fourier formula for the energy (\cite[Eq.~3.5]{mattila}). This proves \eqref{energybound} when $\gamma > (n-1)/2$. The only adjustment required in the case $\gamma \leq (n-1)/2$ is that the first part of \eqref{firstline} must be replaced by $I_t(d_{\#}(\nu, \delta_x ) ) \lesssim \liminf_{j \to \infty} I_t\left( \Delta\left(f_j,g_{j,x}\right) \right)$. \end{proof}

\end{document}